\documentclass[11pt]{article} %

\usepackage{array, enumerate, color, url}
\usepackage{amsmath, epsfig}
\usepackage{latexsym}
\usepackage{mathrsfs}
\usepackage{amssymb}
\usepackage{graphicx}


 
\newtheorem{thm}{Theorem}[section]

\newcommand{\n}{\noindent}

\newtheorem{lemma}[thm]{Lemma}

\newtheorem{cor}[thm]{Corollary}

\newenvironment{proof}{{\bf Proof}.}{\rule{3mm}{3mm}}

\title{Every planar graph without adjacent cycles of length at most $8$ is $3$-choosable}

\author{Runrun Liu  and  Xiangwen Li\thanks{Supported in part by  the NSFC (11571134)}~\thanks{Corresponding author. Email address: xwli68@mail.ccnu.edu.cn.}
\\
\small School of Mathematics $\&$ Statistics\\
\small Central China Normal University, Wuhan 430079, China
}
\date{}

\begin{document}

\maketitle

\begin{abstract}
DP-coloring as a generalization of list coloring was introduced by Dvo\v{r}\'{a}k and Postle in 2017, who proved that every planar graph without cycles from 4 to 8 is 3-choosable, which was conjectured by Borodin {\it et al.} in 2007. In this paper, we prove that every planar graph without adjacent cycles of length at most $8$ is $3$-choosable, which extends this result of Dvo\v{r}\'{a}k and Postle.
\end{abstract}

\section{Introduction}

Coloring is one of the main topics in graph theory.  A {\em proper $k$-coloring} of a graph $G$ is a mapping $f: V(G)\to [k]$ such that $f(u)\ne f(v)$ whenever $uv\in E(G)$, where $[k]=\{1,2,\ldots, k\}$. The smallest $k$ such that $G$ has a $k$-coloring is called the {\em chromatic number} of $G$ and is denoted by $\chi(G)$. List coloring was introduced by Vizing \cite{V76},  and independently Erd\H{o}s, Rubin, and Taylor \cite{ERT79}. A {\em list assignment} of a graph $G=(V,E)$ is a function $L$ that assigns to each vertex $v\in V$ a list $L(v)$ of colors. An {\em $L$-coloring} of $G$ is a function $\lambda:V\to\cup_{v\in V}L(v)$ such that $\lambda(v)\in L(v)$ for every $v\in V$ and $\lambda(u)\ne\lambda(v)$ whenever $uv\in E$. A graph $G$ is {\em $k$-choosable} if $G$ has an $L$-coloring for every assignment $L$ with $|L(v)|\ge k$ for each $v\in V(G)$. The {\em choice number}, denoted by $\chi_l(G)$, is the minimum $k$ such that $G$ is $k$-choosable.

The techniques to approach the list problems are less than those used in ordinary coloring. For ordinary coloring, identifications of vertices are involved in the reduction configurations. In list coloring, since  different vertices have different lists, it is no possible for one to use identification of vertices. With this motivation, Dvo\v{r}\'ak and Postle \cite{DP17} introduced correspondence coloring (or DP-coloring) as a generalization of list-coloring.

 A {\em $k$-correspondence assignment} for $G$ consists of a list assignment $L$ on vertices in $V(G)$  and a function $C$ that assigns every edge $e=uv\in E(G)$ a matching $C_e$ between $\{u\}\times L(u)$ and $\{v\}\times L(v)$.

A {\em $C$-coloring} of $G$ is a function $\phi$ that assigns each vertex $v\in V(G)$ a color $\phi(v)\in L(v)$, such that for every $e=uv\in E(G)$, the vertices $(u,\phi(u))$ and $(v,\phi(v))$ are not adjacent in $C_e$. We say that $G$ is {\em $C$-colorable} if such a $C$-coloring exists.

The {\em correspondence chromatic number} $\chi_{DP}(G)$ of $G$ is the smallest integer $k$ such that $G$ is $C$-colorable for every $k$-correspondence assignment $(L,C)$.

Let $(L,C)$ be a $k$-correspondence assignment for a graph $G$, and let $W=v_1v_2\ldots v_m$ with $v_m=v_1$ be a closed walk of length $m$ in $G$. We say that the assignment $(L,C)$ is {\em inconsistent} on $W$ if there exist $c_i\in L(v_i)$ for $i\in [m]$ such that $(v_i,c_i)(v_{i+1},c_{i+1})$ is an edge of $C_{v_iv_{i+i}}$ for $i\in [m-1]$, and $c_1\ne c_m$. The $(L,C)$ is {\em consistent} if $(L,C)$ is inconsistent on none of the closed walks in $G$.

Recently, DP-coloring  is studied widely. On the one hand, some results of list coloring were generalized to DP-coloring. For this literature, the readers can see   \cite{KO17,KY17,LLNSY18,LLYY18}.  On the other hand,  some new results on DP-coloring are obtained and can be found in \cite{B16,BK17,BKZ17}. We  here pay more attention to  the result of Dvo\v{r}\'ak and Postle \cite{DP17}, who solved a conjecture of  Borodin  (2007) by the technique developed in DP-coloring, as follows.

\begin{thm}[\cite{DP17}]
\label{main0}
Every planar graph $G$ without cycles of lengths $4$ to $8$ is 3-choosable.
\end{thm}

The well-known Four Color Theorem states that all planar graphs
can be colored with four colors such that no adjacent vertices share a common color.   There was much effort to find sufficient conditions for a graph to be 3-colorable.
The classic theorem by Gr\"{o}tzch(\cite{G59}, 1959) shows that planar graphs without 3-cycles are 3-colorable.

In 1969, Havel (\cite{H69}) posed a problem: Does there exists a constant $C$ such that every planar graph with the minimum distance between triangles at least $C$ is 3-colorable? Recently, Dvo\v{r}\'{a}k, Kr\'{a}l' and Thomas~\cite{DKT16} proved that the distance is at least $10^{100}$.
In 1976,  Steinberg (\cite{S76}) conjectured that every  planar graph without 4-cycles and 5-cycles is 3-colorable.
 Erd\H{o}s relaxed Steinberg's conjecture and suggested to determine the smallest integer $k$, if it exists, such that every planar graph without cycles of length from $4$ to $k$ is 3-colorable. The best bound for such $k$ is 7, and it is proved by Borodin, Glebov, Raspaud, and Salavatipour~\cite{BGRS05}. In 2003,
 Strong Bordeaux Conjecture (\cite{BR03}) by Borodin and Raspaud was posed as follows: Every planar graph without 5-cycle and adjacent triangles is 3-colorable.

Recently,  Cohen-Addad, Hebdige, Kral, Li and Salgado~\cite{CHKLS17} presented a counterexample to the Steinberg's Conjecture, as well as to the Strong Bordeaux Conjecture.
Borodin, Montassier and Raspaud \cite{BMR10} asked to determine  the smallest integer $k$ such that every planar graph without adjacent  cycles of length at most $k$ is 3-colorable and proved the following result.

\begin{thm}\label{main00}
Every planar graph without adjacent  cycles of length at most 7 is 3-colorable.
\end{thm}

Motivated by Theorems~\ref{main0} and \ref{main00},  we present the following result in this paper.

\begin{thm}\label{main10}
Every planar graph without adjacent cycles of length at most $8$ is $3$-choosable.
\end{thm}

In the end of this section, we introduce some notations used in the paper. Graphs mentioned in this paper are all simple. For a cycle $K$ of a plane graph $G$, we use $int(K)$ and $ext(K)$ to denote the sets of vertices located inside and outside $K$, respectively. The cycle $K$ is called a {\em separating cycle} if $int(K)\ne\emptyset\ne ext(K)$. Let $V$ and $F$ be the set of vertices and faces of $G$, respectively. For a face $f\in F$, if the vertices on $f$ in a cyclic order are $v_1, v_2, \ldots, v_k$, then we write $f=[v_1v_2\ldots v_k]$. Let $b(f)$ be the vertex set of $f$. A $k$-vertex ($k^+$-vertex, $k^-$-vertex) is a vertex of degree $k$ (at least $k$, at most $k$).  A $k$-face ($k^+$-face, $k^-$-face) is a face whose boundary walk has length of $k$ (at least $k$, at most $k$). The same notation will be applied to walks and cycles.

\section{Lemmas}\label{prelim}

The $C$-coloring was recently introduced by  Dvo\v{r}\'{a}k and Postle \cite{DP17}, who proved
the following nice relationship between choosability and correspondence coloring.

\begin{thm}[\cite{DP17}]\label{main2}
A graph $G$ is $k$-choosable if and only if $G$ is $C$-colorable for every consistent $k$-correspondence assignment $C$.
\end{thm}

Utilizing Theorem~\ref{main2}, the technique used by  Dvo\v{r}\'{a}k and Postle for solving Borodin's conjecture is $C$-coloring.
 We follow this Dvo\v{r}\'ak and Postle's idea and prove the following result.

\begin{thm}
\label{main1}
Every planar graph $G$ without adjacent cycles of length at most $8$ is $C$-colorable for every $3$-correspondence assignment $C$ that is consistent on every closed walk of length $3$ in $G$.
\end{thm}

We actually prove the following result which is a little stronger than Theorem~\ref{main1}.

\begin{thm}\label{main3}
Let $G$ be a plane graph without adjacent cycles of length at most $8$. Let $S$ be a set of vertices of $G$ such that either $|S|=1$, or $S$ consists of all vertices on a face of $G$. Let $C$ be a $3$-correspondence assignment for $G$ such that $C$ is consistent on every closed walk of length $3$ in $G$. If $|S|\le12$, then for every $C$-coloring $\phi_0$ of $G[S]$, there exists a $C$-coloring $\phi$ of $G$ whose restriction to $S$ is $\phi_0$.
\end{thm}

Let $(L,C)$ be a $k$-correspondence assignment on $G$.  An edge $uv\in E(G)$ is {\em straight}  if every $(u,c_1)(v,c_2)\in E(C_{uv})$ satisfies $c_1=c_2$. An edge $uv\in E(G)$ is {\em full} if $C_{uv}$ is a perfect matching.  The following lemma from \cite{DP17} says that an $C$-coloring in certain subgraph $H$ of $G$ is the same as a list coloring, which plays the crucial role in proving results of choosability using DP-coloring.

\begin{lemma}[\cite{DP17}]\label{straight}
Let $G$ be a graph with a $k$-correspondence assignment $C$. Let $H$ be a subgraph of $G$ such that for every cycle $D$ in $H$, the assignment $C$ is consistent on $D$ and all edges of $D$ are full. Then we may rename $L(u)$ for $u\in H$ to obtain a $k$-correspondence assignment $C'$ for $G$ such that all edges of $H$ are straight in $C'$.
\end{lemma}

\n{\bf Proof of Theorem~\ref{main1}}
By Theorem~\ref{main2}, we need to prove  that $G$ is $C$-colorable for arbitrary  $3$-correspondence assignment $C$ such that $G$ is consistent on every closed walk of length $3$ in $G$.   Take $S$ to be an arbitrary vertex in $G$.  By Theorem~\ref{main3}, $G$ is $C$-colorable.

\section{Proof of Theorem~\ref{main3}}\label{redults}

From now on, we always let $C$ be a $k$-correspondence assignment on $G$ that is consistent on every closed walk of length $3$.   Assume that Theorem~\ref{main3} fails, and let $G$ be a minimal counterexample, that is, there exists no $C$-coloring $\phi$ of $G$ whose restriction to $S$ is equal to $\phi_0$ such that
\begin{equation}
 |V(G)| \mbox{ is minimized.}
 \end{equation}
 Subject to (1), the number of edges of $G$ that do not join the vertices of $S$
 \begin{equation}
   |E(G)|-|E(G[S])| \mbox{  is minimized}.
  \end{equation}
Subject to (1) and (2), the total number of edges in the matchings of the $3$-correspondence assignment $C$
 \begin{equation}
 \sum_{uv\in E(G)}|E(C_{uv})| \mbox{ is maximized}.
\end{equation}

When $S$ consists of the vertices of a face, we will always assume that $D$ is the outer face of the embedding of plane graph $G$. And we call a vertex $v$ or a face $f$ {\em internal} if $v\notin D$ or $f\ne D$.

The following Lemma~\ref{counter} to Corollary~\ref{cor} about some crucial properties of the minimal counterexample and the correspondence assignment are basically from \cite{DP17}. For completeness, we include the proofs here.

\begin{lemma}\label{counter}
 Each of the following holds:
\begin{enumerate}[(a)]
\item $V(G)\ne S$;
\item $G$ is $2$-connected;
\item each vertex not in $S$ has degree at least $3$;
\item $G$ does not contain separating $k$-cycle for $3\le k\le12$;
\item $S=V(D)$ and $D$ is an induced cycle.
\item If $P$ is a path of length $2$ or $3$ with both ends in $S$ and no internal vertex in $S$, then no edge of $P$ is contained in a triangle that shares at most one vertex with $S$.
\end{enumerate}
\end{lemma}
\begin{proof}
(a) Suppose otherwise that $V(G)=S$. In this case, $\phi_0$ is a $C$-coloring of $G$, a contradiction.

(b) By the condition (1), $G$ is connected. Suppose otherwise that $v$ is a cut-vertex of $G$. Thus, we may assume that $G=G_1\cup G_2$ such that $V(G_1)\cap V(G_2)=\{v\}$.  If $v\in S$, then by the condition (1)  $G_1$ and $G_2$ have $C$-coloring extending $\phi_0$ such that these $C$-colorings have the same color at $v$.  Thus, $G$ has a $C$-coloring, a contradiction. Thus, assume that $v\notin S$. We assume, without loss of generality, that $S\subseteq V(G_1)$.  By the condition (1), $\phi_0$ can be extended to $\phi_1$ of $G_1$. Then, $\phi_1(v)$ can be extended to $\phi_2$ of $G_2$. Now $\phi_1$ and $\phi_2$ together give an extension of $\phi_0$ to $G$, a contradiction.

(c) Let $v$ be a $2^-$-vertex in $G-S$. By the condition (1), $\phi_0$ can be extended to a $C$-coloring $\phi$ of $G-v$. Then we can extend $\phi$ to $G$ by selecting a color $\phi(v)$ for $v$ such that for each neighbor $u$ of $v$, $(u,\phi(u))(v,\phi(v))\notin E(C_{uv})$, a contradiction.

(d)  Let $K$ be a separating $k$-cycle with $3\le k\le 12$. By the condition (1), $\phi_0$ can be extend to a $C$-coloring $\phi_1$ of $ext(K)\cup K$, and the restriction of $\phi_1$ to $K$ extends to a $C$-coloring $\phi_2$ of $int(K)$. Thus, $\phi_1$ and $\phi_2$ together give a $C$-coloring of $G$ that extends $\phi_0$, a contradiction.

(e) Suppose otherwise that $S=\{v\}$ for some vertex $v\in V(G)$. If $v$ is incident with a $12^-$-cycle $f_1$, we may assume that  $v$ is incident with a $12^-$-face by (d). We now redraw $G$ such that $f_1$ is the outer cycle of $G$ and choose a $C$-coloring $\phi$ on the boundary of $f_1$. Let $S_1=V(f_1)$. In this case, $|E(G)|-|E(G[S_1])|<|E(G)|-|E(G[S])|$. By the condition (2), $G$ has a $C$-coloring that extends the colors of $S_1$, thus $G$ has a $C$-coloring that extends $\phi_0$, a contradiction. Thus, we may assume that all cycles incident with $v$ are $13^+$-cycles.   Let $f_2$ be a $13^+$-face  incident with $v$. Let $v_1$ and $v_2$ be the neighbors of $v$ on $f_2$. Let $G_2=G\cup \{v_1v_2\}$. We redraw $G$ such that $f_2$ is the outer cycle of $G_2$. Let $S_2=\{v,v_1,v_2\}$ and $C_2$ be obtained from $C$ by letting the matching between $v_1$ and $v_2$ be edgeless.  It is easy to verify that $|E(G_2)|-|E(G[S_2])|<|E(G)|-|E(G[S])|$. By the condition (2), $G_2$ has a $C_2$-coloring that extends the colors of $S_2$. This implies that $G$ has a $C$-coloring that extends $\phi_0$, a contradiction again. So $S=V(D)$.

We may assume that $D$ contains a chord $uv$.  By (a) $V(G)\not=S$. Thus $D$ together with the chord $uv$ forms two cycles with common edge $uv$, each of which has  length less than 12 by our assumption that $|S|\leq 12$. By (d), such two cycles are the boundaries of two faces. This means that $S=V(G)$, a contradiction to (a).

(f) Let $P=x_1x_2\ldots x_k$, where $k=3, 4$, $x_1, x_k$ are in $D$ and no internal vertex in $D$. Suppose otherwise that one edge $x_ix_{i+1}$ of $P$ is contained in a triangle $f=[x_ix_{i+1}x]$ which  has at most one common vertex with $D$, where $1\leq i\leq k-1$. Let $P_1$ and $P_2$ be two paths on $D$ between $x_1$ and $x_k$. Then  $D_i=P\cup P_i$ is a cycle for $i=1, 2$. We assume, without loss of generality, that $f$ is inside of $D_1$. Note that $G$ contains no adjacent cycles of length at most $8$. By our assumption, $f$ and $P_1$ have at most one common vertex. Since $D_1$  and $f$ are adjacent, $D_1$ is a $9^+$-cycle. Similarly, $D_2$ is also $9^+$-cycle. Since $|S|\le12$, $9+9\leq |D_1|+|D_2|=|D|+2(k-1)\leq 18$, which implies that $k=4$, $|D_1|=|D_2|=9$, $|S|=12$ and $P$ is a path of length $3$. By (d),   $D_1$ is not a separating $9$-cycle and $x$ is in $D$. This implies that $G$ has adjacent $8^-$-cycles, a contradiction.
\end{proof}

\begin{lemma}\label{matching}
 If $e=uv\notin E(D)$, then $|E(C_{uv})|\ge2$. Moreover, if $e$ is not contained in a triangle, then $|E(C_{uv})|=3$.
\end{lemma}
\begin{proof}
First we show that if $e$ is not contained in a triangle, then $|E(C_{uv})|=3$. Suppose otherwise that $uv$ is not full. Let $C'$ be the 3-correspondence assignment obtained from $C$ by adding edges to $C_{uv}$ so that $uv$ is full and keeping $C_e$ unchanged for every edge $e\not=uv$. In this case, $\sum_{e}|E(C'_e)|>\sum_e|E(C_e)|$, contrary to  the condition (3) of $G$.

We now assume that $uv$ is contained in a triangle $f=[uvw]$. If $C_{uv}=\emptyset$, then let $G'=G-e$.  Thus, $|E(G')|-|E(S)|<|E(G)|-|E(S)|$, contrary to the condition (2) of $G$. So we may assume that $|E(C_{uv})|=1$.  By Lemma~\ref{straight}, we may assume that $C_{uv}=\{(u,1)(v,1)\}$. Define $C^{a, b}$ by adding an edge $(u,a)(v,b)$ to $C_{uv}$ and keeping $C_e$ unchanged for every edge $e\not=uv$, where $a, b\in \{2, 3\}$. In this case, $\sum_{e}|E(C^{a, b}_e)|>\sum_e|E(C_e)|$. On the other hand, if $G$ has a $C^{a, b}$-coloring, then it has a $C$-coloring.   Since $G$ has no $C$-coloring, by the condition (3) of $G$,  there is a  closed $3$-walk bound $f$ which is inconsistent in $C^{a, b}$.

Assume that  $C_{uw}$ has no edge incident with $(u, 2)$. Then the only closed $3$-walk in $C^{2, b}$ may be inconsistent is $uvwu$ for $b\in \{2, 3\}$, and thus $C_{vw}\cup C_{wu}$ contains a path  $(v,b)(w, d_b)(u, c_b)$ for some colors $d_b$ and $c_b$ such that $c_b\not=2$. By symmetry, we may assume that $c_b=1$. This implies that $C$ has a path $(v, 1)(u, 1)(w, d_b)(v, b)$. This means that $C$ is inconsistent on $vuwv$, contrary to our assumption that $C$ is consistent on every closed walk of length 3.  So far, we have proved that  $C_{uw}$ has an edge incidents with $(u, 2)$.

 By symmetry, we may assume that $C_{uw}$ has an edge incident with $(u,3)$, $C_{vw}$ has two independent edges incident with $(v,2)$ and $(v,3)$, respectively. By the pigeonhole principle, there exist colors $c_u, c_v\in \{2, 3\}$ and $c_w$ such that  $C_{vw}\cup C_{wu}$ contains a path $(u, c_u)(w, c_w)(v, c_v)$. This means that $C^{c_u, c_v}$  is consistent on all closed $3$-walks, which is a contradiction. Therefore, add an edge $(u, c_u)(v, c_v)$ to $C$ and denote by $C''$. This leads to that $\sum_{e}|E(C''_e)|>\sum_e|E(C_e)|$, contrary to  the condition (3) of $G$.
\end{proof}

\begin{lemma}\label{333+face}
Let $f=[uvw]$ be a $3$-face in $G$ with $d(u)=d(w)=3$ and $u,w\notin S$. Then each edge of $f$ is full.
\end{lemma}
\begin{proof}
Suppose to the contrary that at least one of all edges of $f$ is not full. Let $u_1$ ($w_1$) be the neighbor of $u$ ($w$) rather than $v$ and $w$ ($v$ and $u$). Applying Lemma~\ref{straight} to subgraph induced by the edge set $\{uu_1, uv, uw, ww_1\}$, these edges are straight in $C$.

Suppose that $C'$ is the $3$-correspondence assignment for $G$ such that  $C'_e=C_e$ for each  $e\in E(G)\setminus E(f)$ and  that all edges of $f$ are straight and full in $C'$.  Since each edge of $f$ is full in $C'$ but not in $C$, $\sum_{e\in E(G)} |E(C'_{e})|>\sum_{e\in E(G)}|E(C_e|$, By the condition (3) of $G$,   $\varphi_0$ can be extended to  a $C'$-coloring $\varphi'$ of $G$. On the other hand, by our assumption,   $G$ is not  $C$-colorable. This contradiction is produced since $C'$ differs from $C$ only on the edges of $f$. Note that all edges of $f$ other than $vw$ are straight in $C$.  By symmetry, we may assume that $\varphi'(v)=1,\varphi'(w)=2$ and $(v,1)(w,2)\in E(C_{vw})$. If $C_{uv}$ has no edge incident with  $(v,1)$, then we can modify $\varphi'$ to a $C$-coloring of $G$ by recoloring $w$ by a color $c$ in $\{1,3\}\setminus \{\varphi'(w_1)\}$ and by recoloring $u$ by a color from $\{1,2,3\}\setminus\{c,\varphi'(u_1)\}$. Thus, we may assume that $C_{uv}$ has one edge incident with $(v,1)$. Since the edge $uv$ is straight in $C$,  $(v,1)(u,1)\in E(C_{uv})$. By Lemma~\ref{matching},  $|E(C_{uw})|\ge2$. So $(u,2)(w,2)\in E(C_{uw})$ or $(u,1)(w,1)\in E(C_{uw})$, which implies that $C$ is not consistent on all closed $3$-walks in $f$,  a contradiction.
\end{proof}

\begin{figure}[ht]
\centering
\includegraphics[scale=0.46]{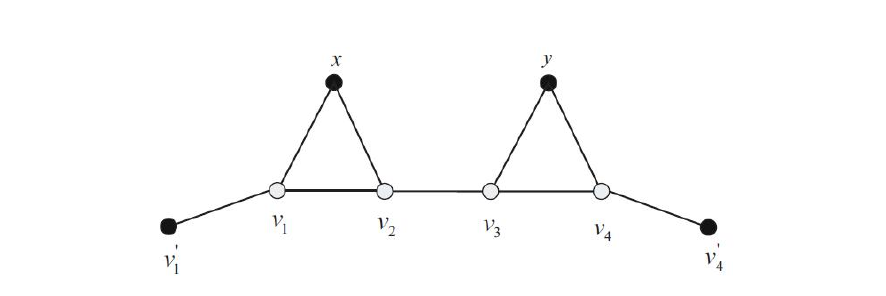}
\caption{A tetrad.}\label{tetrad}
\end{figure}

A {\em tetrad} in a plane graph is a path $v_1v_2v_3v_4$ of vertices of degree three contained in the boundary of a face, such that both $v_1v_2$ and $v_3v_4$ are edges of triangles. See Fig.1.

\begin{lemma}\label{tetrad}
Every tetrad in $G$ contains a vertex of $S$.
\end{lemma}
\begin{proof}
Let $v_1v_2v_3v_4$ be a tetrad in $G$. Let $N(v_1)=\{x,v_2,v_1'\}$ and $N(v_4)=\{y,v_3,v_4'\}$, where $v_1', v_4'$ are not in the tetrad. Since $G$ has no adjacent $8^-$-cycles, the vertices in $\{v_1, v_2, v_3, v_4, v_1', v_4', x, y\}$ are distinct. Suppose otherwise that none of the vertices in the tetrad is in $S$.

Since two edges of the path $xv_2v_3y$ are contained in two triangles,
 at least one of $x$ and $y$ is not in $S$ by Lemma~\ref{counter}(f). Applying Lemma~\ref{counter}(f) to paths $xv_1v_1'$ and  $yv_4v_4'$, respectively,  $|\{v_1',x\}\cap S|\leq 1$ and  $|\{v_4',y\}\cap S|\leq 1$. By symmetry, we may assume that
 \begin{equation}
 \mbox{ $y\notin S$ and either $v_1'\notin S$ or $y$ has no neighbors in $S$.}
\end{equation}
By Lemmas~\ref{matching} and ~\ref{333+face},  all edges incident with the vertices in the tetrad are full. By Lemma~\ref{straight}, they are straight. Let $G'$ be the graph obtained by identifying $y$ and $v_1'$ of $G-\{v_1,v_2,v_3,v_4\}$ and let $C'$ be the restriction of $C$ to $E(G')$. By our assumption (4), the identification does not create an edge between vertices of $S$, and thus $\phi_0$ is also a $C'$-coloring of the subgraph of $G'$ induced by $S$. If there were a path $Q$ of length at most 8 between $y$ and $v_1'$, then $G$ would have a cycle $K$ of length at most 12 which is obtained from the path $yv_3v_2v_1v_1'$  and $Q$. Note that $K$ is a separating cycle, contrary to Lemma~\ref{counter}(d). Thus, no new $8^-$-cycles are created in $G'$. This implies that $G'$ contains no adjacent cycles of length at most $8$. By the argument above, the identifying $y$ and $ v_1'$ does not create any new triangle. Thus, $C'$ is consistent on all closed walks of length three in $G'$. Since $|V(G')|<|V(G)|$,  $\phi_0$ can be extended to  a $C'$-coloring $\phi$ of $G'$  by (1). We extend $\phi$ to a $C$-coloring of $G$ by coloring $v_1'$ and $y$ with the color of the identifying vertex and then color $v_4$ and $v_3$ in order with colors distinct from the colors used by their neighbors. Finally, we  color $v_1$ and $v_2$ as follows. Note that $v_1'$ and $v_3$ have different colors and all edges incident with $v_1$ and $v_2$ are straight and full. If the color of $x$ is the same  color of $v_1'$, color $v_2$ and then color $v_1$; if the color of $x$ is the same color of $v_3$, color $v_1$ and then color $v_2$; if the colors of $v_3, x$ and $v_1'$ are distinct, assign the colors of $v_3$ and $v_1'$ to $v_1$ and $v_2$, respectively,  a contradiction.
\end{proof}

\medskip

From now on,  let $F_k=\{f: \text{ $f$ is a $k$-face and } b(f)\cap D=\emptyset\}$,  $F_k'=\{f:  \text{ $f$ is a $k$-face and } |b(f)\cap D|=1\}$.
A $k$-face $f$ is {\em special} if $f\in F_k'$, where $3\leq k\leq 8$. An internal $3$-vertex   is {\em bad} if it is incident with a 3-face $f\notin F_3'$, {\em light} if it is either incident with a 3-face $f\in F_3'$ or with a 4-face from $F_4$ or 5-face from  $F_5$, {\em good} if it is neither bad nor light.

\begin{cor}\label{cor}
No face of $G$ is incident with five consecutive bad vertices. Furthermore, if a face of $G$ is incident with consecutive vertices $v_0,v_1,\ldots,v_5$ and the vertices $v_1,\ldots,v_4$ are bad, then the edges $v_0v_1,v_2v_3$ and $v_4v_5$ are incident with triangles.
\end{cor}
\begin{proof}
Suppose otherwise that $f$ is incident with five consecutive bad vertices $v_1,\ldots,v_5$. Since $v_3$ is on a $3$-face and $G$ contains no adjacent $3$-faces, either $v_2v_3$ or $v_3v_4$ is an edge on a $3$-face. In the former case, $v_4v_5$ is an edge on a 3-face; in the latter case,  $v_1v_2$ must be an edge of a $3$-face. This implies that either $v_2 v_3v_4v_5$ or $v_1v_2v_3v_4$ is a tetrad. On the other hand, none of $v_1, \ldots, v_5$ is in $S$ since each of them is a bad vertex. This contradicts  Lemma~\ref{tetrad}.

Similarly, if a face of $G$ is incident with consecutive vertices $v_0,v_1,\ldots,v_5$ and the vertices $v_1,\ldots,v_4$ are bad, then either each of $v_1v_2,v_3v_4$ is an edge of a 3-face or each of $v_0v_1,v_2v_3,v_4v_5$ is an edge incident with a triangle. In  the former case, the path $v_1v_2v_3v_4$ is a tetrad and each of $v_1,\ldots, v_4$ is a bad vertex, contrary to Lemma~\ref{tetrad}.
\end{proof}

\medskip

The following two lemmas play key roles in the proof of  Theorem~\ref{main3}.

\begin{lemma}\label{counternew}
Each internal $4$-face is incident with at most two  vertices of the outer face $D$. Furthermore, if it is indeed incident with two vertices $u,v$ of $D$, then $uv\in E(D)$.
\end{lemma}
\begin{proof}
Suppose otherwise that an internal $4$-face $f=[uvwx]$  is incident with at least three vertices of the outer face $D$. By Lemma~\ref{counter}(e), $D$ has no chords. Thus,  $f$ and $D$ have exactly three common vertices and share a common $2$-vertex $u$ and $w$ is an internal vertex.  By Lemma~\ref{counter}(c) $d(w)\ge3$. Suppose that $w$ has a neighbor on $D$ other than $v$ and $x$. Let $f_1$ and $f_2$ be two cycles sharing $wx$ and $wv$ with $f$, respectively.    Since $D$ is a $k$-face with $9\le k\le12$, at least one of $f_1$ and $f_2$ is a $7^-$-face which is adjacent to $f$, contrary to our assumption that  $G$ contains no adjacent cycles of length at most $8$. Thus, $w$ has no other neighbors on $D$ other than $v$ and $x$.  Let $P$ be the longer path on $D$ joining $v$ and $x$. Then $xPvwx$ is a separating $12^-$-cycle in $G$, contrary to Lemma~\ref{counter}(d).

Next, suppose otherwise that $f$ and $D$ share exactly two non-consecutive vertices. Since $D$ is a $12^-$-face, $f$ must be adjacent a  $8^-$-cycle in $G$, a contradiction.
\end{proof}

\medskip

 A $9$-face $f=[v_1v_2\ldots v_9]$ in $F_9$ is {\em special} if each vertex in $\{v_1,v_2\}$ is a light 3-vertex and each vertex in $\{v_3,v_4,v_5,v_7,v_8,v_9\}$ is a bad $3$-vertex.

\begin{lemma}\label{special}
$G$ does not contain any special $9$-face.

\end{lemma}
\begin{proof}
 Suppose otherwise that $f=[v_1v_2\ldots v_9]$ is a special $9$-face in $G$. We claim that $v_1$ and $v_2$ are not on a 3-face. Suppose otherwise. Since each of $\{v_3,v_4,v_5,v_7,v_8,v_9\}$ is a bad $3$-vertex, $v_3$ and $v_4$ are on a 3-face, $v_8$ and $v_9$ are on a 3-face, $v_5$ and $v_6$ are on a 3-face and $v_6$ and $v_7$ are on a 3-face. This implies that the subgraph induced by $v_1, v_2, v_3$ and $v_4$ is a tetrad.
By Lemma~\ref{tetrad}, one of $v_1, v_2, v_3$ and $v_4$ is in $D$, contrary to our assumption that $f$ is a special 9-face. Thus,  $v_1$ and $v_2$ must be on a truly internal $4$-or $5$-face, and each of $\{v_3v_4, v_5v_6, v_6v_7, v_8v_9\}$ is on a $3$-face. Let $v_2'$ be the neighbor of $v_2$ not on $f$.

By Lemmas~\ref{matching} and ~\ref{333+face},  all edges incident with $v_1, v_2,v_3$ and $v_4$ are full. By Lemma~\ref{straight}, they are straight. Let $G'$ be the graph by identifying $v_2'$ and $v_5$ of $G-\{v_1,v_2,v_3,v_4\}$ and let $C'$ be the restriction of $C$ to $E(G')$. Note that all the neighbors of each vertex on $f$ are internal. So the identification does not create an edge between vertices of $S$, and thus $\phi_0$ is a $C'$-coloring of the subgraph of $G'$ induced by $S$. If there were a path $Q$ of length at most 8 between $v_2'$ and $v_5$, then $G$ would have a cycle $K$ of length at most 12 which is obtained from the path $v_2'v_2v_3v_4v_5$  and $Q$. Note that $D$ is a separating cycle, contrary to Lemma~\ref{counter}(d). Thus, no new $8^-$-cycles are created in $G'$. This implies that $G'$ contains no adjacent cycles of length at most $8$ and $C'$ is consistent on all closed walks of length three in $G'$. Since $|V(G')|<|V(G)|$,  $\phi_0$ can be extended to  a $C'$-coloring $\phi$ of $G'$  by (1). We extend $\phi$ to a $C$-coloring of $G$ by coloring $v_2'$ and $v_5$ with the color of the identifying vertex and then coloring $v_1$ and $v_2$ in order with colors distinct from the colors used by their neighbors. Finally, we can color $v_3$ and $v_4$. Note that $v_2$ and $v_5$ have different colors and all edges incident with $v_3$ and $v_4$ are straight and full. Let $v_{34}$ be the common neighbor of $v_3$ and $v_4$. If the color of $v_{34}$ is the same  color of $v_2$, color $v_4$ and then color $v_3$; if the color of $v_{34}$ is the same color of $v_5$, color $v_3$ and then color $v_4$; if the colors of $v_{34}, v_2$ and $v_5$ are distinct, assign the colors of $v_2$ and $v_5$ to $v_4$ and $v_3$, respectively,  a contradiction.
\end{proof}

\medskip


We are now ready to present a discharging procedure that will complete the proof of the Theorem~\ref{main3}.  Let each vertex $v\in V(G)$ have an initial charge of $\mu(v)=d(v)-4$, and each face $f\not=D$ in our fixed plane drawing of $G$ have an initial charge of $\mu(f)=d(f)-4$. Let $\mu(D)=d(D)+4$. By Euler's Formula, $\sum_{x\in V\cup F}\mu(x)=0$.

Let $\mu^*(x)$ be the charge of $x\in V\cup F$ after the discharge procedure. To lead to a contradiction, we shall prove that $\mu^*(x)\ge 0$ for all $x\in V\cup F$ and $\mu^*(D)$ is positive.

\noindent The discharging rules:

\begin{enumerate}[(R1)]
\item \label{3f} Each non-special $3$-face receives $\frac{1}{3}$ from each incident internal vertex.

\item \label{2vertex} Each $2$-vertex receives $\frac{2}{3}$ from each incident internal face.

\item\label{3vertex} Each internal $3$-vertex receives $\frac{2}{3}$ and $\frac{1}{2}$ from each incident $9^+$-face if it is bad and  light $3$-vertex, respectively; each good $3$-vertex receives $\frac{1}{3}$ from each incident face.

\item \label{4vertex} Each internal $4$-vertex $v$ receives $\frac{1}{3}$ from the incident $9^+$-face  which is not adjacent to the non-special 3-face if $v$ is incident with one non-special $3$-face and three $9^+$-faces, $\frac{1}{3}$ from each incident $9^+$-face if $v$ is incident with  two non-special $3$-faces and two $9^+$ faces, $\frac{1}{6}$ from each incident  $9^+$-face if $v$ is incident with one non-special $3$-face, one $8^-$-face and two $9^+$-faces.

\item \label{innerface} After (R\ref{3f}) to (R\ref{4vertex}) each face sends its remaining positive charge to the outer face $D$.

\item \label{verticesonD} Each $3^+$-vertex on $D$ sends $1$ to each incident special  $8^-$-face, $\frac{1}{3}$ to each other incident $8^-$-face other than $D$.

\item \label{outerface}The outer face D sends $\frac{4}{3}$ to each incident $2$-vertex or $3$-vertex  incident with an internal  $8^-$-face, and $1$ to each other incident vertex.

\end{enumerate}

\begin{lemma}\label{VERTEXCHECK}
Every vertex $v$ in $G$ has nonnegative final charge.
\end{lemma}
\begin{proof}
We first consider the final charge of the vertices on $D$. If $d(v)=2$, then  $v$ receives $\frac{2}{3}$ from the incident internal face and $\frac{4}{3}$ from $D$ by (R\ref{2vertex}) and (R\ref{outerface}). So $\mu^*(v)=-2+\frac{2}{3}+\frac{4}{3}=0$. Let $d(v)=3$. By our assumption, $v$ is at most incident with one non-special internal $8^-$-face. If $v$ is incident with an internal $8^-$-face, then  $v$ receives $\frac{4}{3}$ from $D$ and sends out $\frac{1}{3}$ by (R6) and (R7). If $v$ is not incident with any internal $8^-$-face, then $v$ receives $1$ from $D$ by  (R\ref{outerface}). Thus,  $\mu^*(v)=-1+min\{\frac{4}{3}-\frac{1}{3}, 1\}=0$.
If $d(v)=4$, then $v$ gives $1$ to each incident special $8^-$-face, $\frac{1}{3}$ to each non-special $8^-$-face and receives $1$ from $D$ by (R\ref{verticesonD}) and (R\ref{outerface}). So $\mu^*(v)\ge4-4+1-max\{1,\frac{1}{3}\times2\}=0$.
If $d(v)\ge5$, then $v$ gives at most $1$ to each of the $\lfloor \frac{d(v)}{2}\rfloor$ incident internal $8^-$-faces and receives $1$ from $D$ by (R\ref{verticesonD}) and (R\ref{outerface}). So $\mu^*(v)\ge d(v)-4+1-\lfloor \frac{d(v)}{2}\rfloor\ge0$.

Now we consider the final charge of the vertices not on $D$. By Lemma~\ref{counter}(c) $d(v)\ge3$. If $d(v)=3$, then $v$ is incident with at least two $9^+$-faces. If $v$ is not incident with any 3-face, then $v$ gets $\frac{1}{3}$ from each incident face by (R3). If $v$ is incident non-special 3-face, then $v$ receives either $\frac{2}{3}$ from each $9^+$-face and sends $\frac{1}{3}$ to the non-special 3-face by (R3) and (R1). If $v$ is incident with special 3-face, then $v$ receives $\frac{1}{2}$ from each incident $9^+$-face by (R3).  Thus, $\mu^*(v)=-1+min\{\frac{2}{3}\times2-\frac{1}{3}, \frac{1}{2}\times2, \frac{1}{3}\times3\}=0$. If $d(v)=4$, then let $f_i$  be the four incident faces of $v$ in clockwise order, where $1\le i\le4$. If $v$ is incident with no non-special 3-faces, then $\mu^*(v)=\mu(v)=0$ by (R4). If $v$ is incident with exactly one non-special 3-face, say $f_1$, then by (R\ref{4vertex}) $v$ receives either  $\frac{1}{6}$ from each of $f_2$ and $f_4$ if $f_3$ is a $8^-$-face or $\frac{1}{3}$ from $f_3$ if $f_3$ is a $9^+$-face. If $v$ is incident with two non-special 3-faces, say $f_1$ and $f_3$, then $v$ receives $\frac{1}{3}$ from each of $f_2$ and $f_4$ by (R\ref{4vertex}). By (R\ref{3f}) $v$ gives $\frac{1}{3}$ to each incident non-special $3$-face. Thus, $\mu^*(v)=4-4+min\{\frac{1}{6}\times2-\frac{1}{3}, \frac{1}{3}-\frac{1}{3}, \frac{1}{3}\times2-\frac{1}{3}\times2\}=0$. If $d(v)\ge5$, then  $v$ gives at most  $\frac{1}{3}$ to each incident $3$-face by (R\ref{3f}). Thus, $\mu^*(v)\ge d(v)-4-\frac{1}{3}\times\lfloor \frac{d(v)}{2}\rfloor>0$.
\end{proof}

\begin{lemma}\label{FACECHECK}
Each face $f$ other than $D$ has nonnegative final charge.
\end{lemma}
\begin{proof}
Let $\mu'(f)$ be the charge of $f$ after (R\ref{3f}) to (R\ref{4vertex}). By (R\ref{innerface}) we only need to show that $\mu'(f)\ge0$. Assume that $d(f)=3$. By Lemma~\ref{counter}(e), $f$ is incident with no $2$-vertices. If $f\in F_3'$,  then $f$  receives 1 from the incident vertex  from $D$ by (R\ref{verticesonD}). If $f\notin F_3'$, then $f$ receives $\frac{1}{3}$ from each incident vertex by (R\ref{3f}) and (R\ref{verticesonD}).   Thus, $\mu'(f)=3-4+min\{1,\frac{1}{3}\times3\}=0$.

 Assume that $4\le d(f)\le8$.  Since $G$ contains no adjacent faces of length at most $8$, $f$ is not adjacent to any $3$-faces. Assume first that $b(f)\cap D=\emptyset$. Let $d(f)\in\{4,5\}$. If $f$ sends out the charge,  then it is incident with 3-vertices by (R1) and (R3). If $f$ is incident with an internal 3-vertex $v$, then $v$ is light by the definition, and so $f$ can not send charge to $v$ by (R3).  Thus, $\mu'(f)\ge d(f)-4\ge0$. If $d(f)\in\{6,7,8\}$, then $f$ sends out at most $\frac{1}{3}$ to each incident vertex by (R\ref{3vertex}). So $\mu'(f)\ge d(f)-4-\frac{1}{3}d(f)\ge0$.
 Thus, we may assume that $b(f)\cap D\ne\emptyset$.  By (R\ref{verticesonD}) $f$ receives $1$ from the incident vertex on $D$ if it is special, $\frac{1}{3}$ from each incident $3^+$-vertex on $D$ if it is not special. On the other hand, by (R\ref{3vertex})$f$ sends at most $\frac{1}{3}$ to each incident internal vertex. If $d(f)=4$, then by Lemma~\ref{counternew} $f$ shares either exactly one vertex or one edge with $D$. So  $\mu'(f)\ge4-4+min\{1-\frac{1}{3}\times3, \frac{1}{3}\times2-\frac{1}{3}\times2\}=0$. Let $t$ be the number of $2$-vertices on $f$. Let $5\le d(f)\le8$.
 If $t=0$, then $\mu^*(f)\ge d(f)-4+\min\{1-\frac{1}{3}(d(f)-1), 2\times \frac{1}{3}-\frac{1}{3}(d(f)-2)\}>0$.  If $t\ge1$, then $f$ is incident with at least two $3^+$-vertices from $D$, each of them gives $\frac{1}{3}$ to $f$ by (R\ref{verticesonD}). So $\mu^*(f)\ge d(f)-4+\frac{1}{3}\times2-\frac{2}{3}t- \frac{1}{3}(d(f)-t-2)\ge\frac{1}{3}(d(f)-5)\ge0$ since $t\le d(f)-3$.

 Next assume that $d(f)\ge9$. If $f$ is incident with a $2$-vertex, then $f$ is incident with at least two $3^+$-vertices from $D$. By (R6),  each of these two $3^+$-vertices receives nothing from $f$. By (R2),(R3) and (R4), each 2-vertex on $D$ receives $\frac{2}{3}$ from $f$ and each internal vertex of $f$ receives at most $\frac{2}{3}$ from $f$. Thus, $\mu'(f)\ge d(f)-4-\frac{2}{3}(d(f)-2)=\frac{1}{3}(d(f)-8)>0$. It remains for us to consider  that $f$ is incident with no $2$-vertices.

By Corollary~\ref{cor}, $f$ is incident with at most $(d(f)-2)$ bad 3-vertices. By (R\ref{3vertex}) and (R\ref{4vertex}),  each $9^+$-face sends $\frac{2}{3}$ to each incident bad $3$-vertex, $\frac{1}{2}$ to each incident light $3$-vertex and at most $\frac{1}{3}$ to each other incident vertex.   Let $s$ and $t$ be the number of bad $3$-vertices and light $3$-vertices of $f$, respectively. If $d(f)\ge11$, then $\mu'(f)\ge d(f)-4-\frac{2}{3}(d(f)-2)-\frac{1}{2}\times2\ge0$.

Assume  that $d(f)=10$. If $s\le6$, then $\mu'(f)\ge d(f)-4-\frac{2}{3}\times6-\frac{1}{2}\times4=0$ by (R\ref{3vertex}). Let $s=7$. Since a light 3-vertex is either in 3-face from $F_3'$ or on a 4-face from $F_4$ or a 5-face from $F_5$, $t\leq 2$. Thus, $\mu'(f)\ge10-4-\frac{2}{3}\times7-\frac{1}{2}\times2-\frac{1}{3}=0$ by (R\ref{3vertex}) and (R4).
If $s=8$, then the eight bad 3-vertices must be divided into two parts each of which consists of consecutive 4 bad 3-vertices by the other vertices of $f$ by Corollary~\ref{cor}. Thus, the two non-bad vertices cannot be light $3$-vertices. By (R\ref{3vertex}) and (R4), $\mu'(f)\ge10-4-\frac{2}{3}\times8-\frac{1}{3}\times2=0$.

We are left to consider that $d(f)=9$. If $s\le3$, then $\mu'(f)\ge9-4-\frac{2}{3}\times3-\frac{1}{2}\times6=0$. If $s=4$, then $t\le4$. So $\mu'(f)\ge 9-4-\frac{2}{3}\times4-\frac{1}{2}\times4-\frac{1}{3}=0$.
If $s=5$, then $t\leq 3$. In fact, if $t=4$, then each vertex of $f$ is of degree 3. This implies that a bad 3-vertex and a light 3-vertex are in a 3-face, contrary to the definition of bad 3-vertex. If $t\le2$, then $\mu'(f)\ge9-4-\frac{2}{3}\times5-\frac{1}{2}\times2-\frac{1}{3}\times2=0$ by (R3) and (R4). If $t=3$, then $f$ is incident with a $4^+$-vertex which is on a non-special $3$-face and a $5^-$-face which is not a non-special 3-face, thus gets at most $\frac{1}{6}$ from $f$ by (R\ref{4vertex}). So $\mu'(f)\ge9-4-\frac{2}{3}\times5- \frac{1}{2}\times3-\frac{1}{6}=0$ by (R3) and (R4). Let $s=6$. If any two of the three non-bad vertices on $f$ are not adjacent, then none of them is a light $3$-vertex. So  $\mu'(f)\ge9-4-\frac{2}{3}\times6-\frac{1}{3}\times3=0$ by (R3) and (R4). Let $f=[v_1v_2\ldots v_9]$. So we assume that there are two consecutive non-bad vertices $v_1$ and $v_2$ on $f$. By symmetry and Corollary~\ref{cor}, either $v_7$ or $v_6$ is a non-bad vertex. In the former  case, each of $\{v_2v_3,v_4v_5,v_6v_7\}$ is on a $3$-face by Corollary~\ref{cor}. If $v_8v_9$ is on a $3$-face, then $v_7$ is a not light 3-vertex. In this case, none of $v_1$ and $v_2$ is a light 3-vertex.  By (R3) and (R4), $\mu'(f)\ge9-4-\frac{2}{3}\times6- \frac{1}{3}\times3=0$. Thus, assume that $v_8v_9$ is not an edge of a 3-face. In this case, $v_7v_8$ and $v_1v_9$ are on $3$-faces. This implies that none of $v_1,v_2,v_7$ are light $3$-vertices. By (R3) and (R4), we conclude similarly that $\mu'(f)\ge9-4-\frac{2}{3}\times6- \frac{1}{3}\times3=0$. In the latter case, since each of  $v_5$ and $v_7$ is a bad 3-vertex, $v_6$ cannot be a light $3$-vertex. If none of $v_1$ and $v_2$ is light $3$-vertices, then  $\mu'(f)\ge9-4-\frac{2}{3}\times6- \frac{1}{3}\times3=0$ by (R3) and (R4). If exactly one, say $v_1$, of them is a light $3$-vertex, then $v_1v_2$ is an edge of a 3-face or a 4- or 5-face. Since $v_2$ is not a light 3-vertex, $v_2$ is a $4^+$-vertex. Thus, by (R3) and (R4), $f$ sends at most $\frac{1}{2}$, $\frac{1}{6}$, $\frac{1}{3}$ to $v_1$, $v_2$, $v_6$ respectively.  Then $\mu'(f)\ge9-4-\frac{2}{3}\times6-\frac{1}{2}-\frac{1}{6}-\frac{1}{3}=0$. If both of them are light $3$-vertices, then $f$ is a special $9$-face which contradicts Lemma~\ref{special}. The other possibility is that $v_6 \in D$ in which case $v_6$ does not receive charge from $f$. In this case, $\mu'(f)\ge9-4-\frac{2}{3}\times6-\frac{1}{2}\times 2=0$. If $s=7$, then  the seven bad vertices must be divided into two parts, one consisting of consecutive 3 bad 3-vertices; the other consisting of consecutive 4 bad 3-vertices by the other vertices of $f$ by Corollary~\ref{cor}.  We assume, without loss of generality, that none of $v_1$ and $v_5$ is a bad 3-vertex and each of $v_1v_9$, $v_8v_7$ $v_5v_6$ is an edge on a 3-face. By symmetry, assume that $v_1v_2$ and $v_3v_4$ are edges on two different 3-faces. In this case,  $v_5$ is incident with at least three $9^+$-faces. Since $v_1$ is a $4^+$-vertex, it is not a light 3-vertex.  Thus, $v_5$ gets nothing from $f$ and $\mu'(f)\ge9-4-\frac{2}{3}\times7-\frac{1}{3}=0$ by (R3) and (R\ref{4vertex}).
\end{proof}

\medskip

\n{\bf Proof of Theorem~\ref{main3}}. By Lemmas~\ref{VERTEXCHECK} and ~\ref{FACECHECK}, it is sufficient for us to check that the outer face $D$ has positive final charge. By (R\ref{outerface}) $D$ sends each incident vertex at most $\frac{4}{3}$. Thus, $\mu^*(D)\ge d(D)+4-\frac{4}{3}d(D)=\frac{1}{3}(12-d(D))\ge0$. On the other hand, $\mu^*(D)=0$ if and only if $d(D)=12$ and each vertex on $f$ receives $\frac{4}{3}$ from $D$  and $D$ does not get charge by (R5). Therefore, each vertex on $D$ is either a $2$-vertex or a $3$-vertex incident with an internal $8^-$-face. Since $D\not=G$, $D$ is incident with at least one $3$-vertex. In this case, $v$ is incident with a $9^+$-face $f_1$ other than $D$. Furthermore, $f_1$ is incident with at least two $3$-vertices on $D$, which receive no charge from $f_1$ by (R6). Thus, $f_1$ remains at least $d(f_1)-4-\frac{2}{3}(d(f_1)-2)=\frac{1}{3}(d(f_1)-8)\ge\frac{1}{3}$  after sending its charge to all vertices on $f_1$. By (R\ref{innerface}), $D$ gets at least $\frac{1}{3}$ from $f_1$, which implies that $D$ has a positive final charge.
\vskip 1cm

\n{\bf Acknowledgments}

\medskip

The authors  thank  Gexin Yu for the
valuable suggestions which improve the presentation.

\end{document}